 \documentclass[a4paper,1p,10pt]{elsarticle}  

\usepackage{algorithm}
\usepackage{mathrsfs}
\usepackage{algorithmic}

\usepackage{dsfont}

\makeatletter
\def\ps@pprintTitle{%
 \let\@oddhead\@empty
 \let\@evenhead\@empty
 \def\@oddfoot{\centerline{\thepage}}%
 \let\@evenfoot\@oddfoot}
\makeatother

\usepackage{amsmath}
\usepackage{amsfonts}
\usepackage{amsthm}
\usepackage{amssymb}
\usepackage{bm}
\usepackage{cancel}
\usepackage{natbib}
\usepackage{graphicx}

\usepackage{tabularx}
\newcolumntype{C}{>{\centering}X}
\usepackage{gensymb}
\usepackage{ulem}
\usepackage{empheq}
\usepackage{lineno}
\usepackage{bbold}
\usepackage{float}
\usepackage{color}
\usepackage{adjustbox}
\usepackage{ulem}
\usepackage[colorinlistoftodos]{todonotes}

\def\-{\raisebox{.75pt}{-}}

\newcommand{\mydots}{\ifmmode\mathinner{\ldotp\kern-0.2em\ldotp\kern-0.2em\ldotp}\else.\kern-0.13em.\kern-0.13em.\fi}
\newlength\nextcharwidth

\newcommand{\var}{\mathbb{V}}
\newcommand{\esp}{\mathbb{E}}

\newcommand{\ZYB}{ZY\!B}
\DeclareMathOperator*{\argmax}{arg\,max} 

\newcommand{\mathsout}[1]
{\bgroup\mathchoice
  {\sbox0{$\displaystyle{#1}$}%
    \usebox0\hspace{-\wd0}%
    \rule[0.5\ht0-0.5\dp0-.5pt]{\wd0}{1pt}}%
  {\sbox0{$\textstyle{#1}$}%
    \usebox0\hspace{-\wd0}%
    \rule[0.5\ht0-0.5\dp0-.5pt]{\wd0}{1pt}}%
  {\sbox0{$\scriptstyle{#1}$}%
    \usebox0\hspace{-\wd0}%
    \rule[0.5\ht0-0.5\dp0-.5pt]{\wd0}{1pt}}%
  {\sbox0{$\scriptscriptstyle{#1}$}%
    \usebox0\hspace{-\wd0}%
    \rule[0.5\ht0-0.5\dp0-.5pt]{\wd0}{1pt}}%
\egroup}

\newlength\myindent
\newtheorem{proposition}{Proposition}
\makeatletter
\renewcommand{\fnum@algorithm}{\fname@algorithm}
\makeatother

\begin{document}
\title{Inverting Regional Sensitivity Analysis to reveal sensitive model behaviors}









\author[mis]{S\'ebastien Roux \footnote{corresponding author email : sebastien.roux@inrae.fr}}

\author[mis]{Patrice Loisel}

\author[emm]{Samuel Buis}

\address[mis]{MISTEA, Univ Montpellier, INRAE, Institut Agro, Montpellier, France}
\address[emm]{EMMAH, INRAE, Universit\'e d'Avignon et des Pays de Vaucluse, Avignon, France}

\maketitle



\section*{Abstract}
{

We address the question of sensitivity analysis for model outputs of any dimension using Regional Sensitivity Analysis (RSA). Classical RSA computes sensitivity indices related to the impact of model inputs variations on the occurrence of a target region of the model output space. In this work, we invert this perspective by proposing to find, for a given target model input, the region whose occurrence is best explained by the variations of this input. When it exists, this region can be seen as a model behavior which is particularly sensitive to the variations of the model input under study. We name this method iRSA (for inverse RSA).
\newline
iRSA is formalized as an optimization problem using region-based sensitivity indices  and solved using dedicated numerical algorithms. Using analytical and numerical examples, including an environmental model producing time series, we show that iRSA can provide a new graphical and interpretable characterization of sensitivity for model outputs of various dimensions.
}

\vspace{1cm}
\textit{Keywords}: Multivariate sensitivity analysis, Target sensitivity analysis,  Sobol' indices, Cluster-based GSA, Graphical sensitivity analysis

\newpage

\section{Introduction}

The analysis of models with complex outputs (temporal, spatial, heterogeneous) is one of the current challenges of sensitivity analysis  \cite{marrel2016sensitivity,saltelli2021sensitivity}.
There are two main approaches to handle multidimensional outputs. The first one aims at computing a set of scalar indices associated to coordinates along a basis of functions on which the model outputs are projected (\cite{campbell06,lamboni2009multivariate,xiao2018multivariate}). Its flexibility and interpretability are limited by the choice of the basis. The second approach summarizes the impact of the model inputs on the  variability of all the model outputs using a single index \cite{lamboni2011multivariate,gamboa2014,lamboni2019multivariate,xu2019generalized}. It provides useful aggregated indices but does not allow a fine understanding of model outputs sensitivity.

Spear and Hornberger  introduced in \cite{spear1980eutrophication} the concept of model behavior in sensitivity analysis through the Regional Sensitivity Analysis (RSA) approach. The principle is to start from the definition of a target region of the output space (denoted as "behavioral") and to analyze the impact of the variations of model inputs on its occurrence. Using model behaviors expressed as regions of the output space appears to be an efficient method to get interpretable characterizations of model properties. It has also the property to scale to any dimension of the output space \cite{roux2021cluster}. Among the last developments on RSA, two are of particular interest in the present study:
i) the application of RSA  in the context of reliability engineering to characterize parameter sensitivity in relation to a critical domain of the output space (e.g. the failure domain of a system) using sensitivity measures compatible with rare events and taking into account interactions (Target SA, \cite{marrel2021statistical,idrissi2021developments}), ii) its application in combination with a clustering procedure in order to characterize parameter sensitivity with respect to the dominant model behaviors detected in the model output space (Distance-Based Generalized SA \cite{fenwick2014quantifying}, Cluster-based GSA, \cite{roux2021cluster}). 

These approaches  rely on an a priori characterization of the behaviors (regions of the output space) to be analyzed. They are identified either by experts or automatic clustering of the simulations. 
In this work, we propose a new perspective on the link between behaviors and sensitivity analysis based on an extension of the above mentioned developments on RSA. Instead of trying to a priori characterize target regions of the output space, we propose to use an optimization procedure in order to reveal the regions of the output space the most sensitive to the variations of a given input, i.e. whose occurrences are the best explained by the variations of this input. We name this approach iRSA (for inverse Regional Sensitivity Analysis). The formalization of iRSA in terms of principle and numerical algorithms will be presented in Section \ref{sec_method}. The results of the method application are presented in Section \ref{sec_results} on three examples: a model with scalar outputs which is analytically solved and two models on which numerical algorithms are tested, a model with 2D outputs and an environmental model producing time series.

\section{iRSA methods}
\label{sec_method}
In this section, we formalize iRSA in a general context, with a focus on the  sensitivity-based optimization criteria that can be used and on their minimization in the context of multivariate outputs.

\subsection{Notations}
We consider a model $f$ whose inputs are noted $X_i$, with $i \in [1,n]$ and have independent distributions. The model output is noted $Y=f(X_1,..,X_n)$ and can be multivariate.

Model behaviors in the context of RSA are defined as regions of the model output space. Depending on the formalization context, these regions are continuous sets (as we will see in Section \ref{sec_res_1d}) or discrete set (in Sections \ref{sec_res_2d} and \ref{sec_res_nd}). In the latter case, they are ensembles of simulated points and we will denote them as clusters of model outputs. In the following, we generically denote by $\mathcal{C} $ the  model output space.

\subsection{iRSA principle}
The principle of iRSA is to find for each input factor $X_i$ a partition of the  model output space $\mathcal{C}$ that maximizes a cost function related to parameter sensitivity. We will be interested in either finding the region $C$ whose occurrence is the most influenced by a given input, in which case iRSA will provide the partition $(C,\bar{C})$ (where $\bar{C}$ is the complement of $C$), or the two non-overlapping regions $C$ and $C'$ whose transition from the one to the other is the most influenced by a given input, in which case iRSA will provide the partition made of three regions $(C,C',\overline{C \cup C'})$.

\subsection{Main criteria and associated interpretations}
We consider the region-based sensitivity indices introduced in \cite{roux2021cluster}. These indices use membership functions which characterize the level of membership of any point of the output space to a given region $C$. In this work, we consider binary membership functions that can be expressed using the indicator function $\mathbb{1}_C(.)$. 

The idea behind these indices is simply to build sensitivity indices from the standard Sobol' indices \cite{sobol1993,saltelli2000sensitivity} obtained when transforming the multivariate output into a scalar one using the  membership functions.
\subsubsection{Indices based on single region membership}
The first and simplest region-based indices that can be used for an input $X_i$ are the Sobol' indices (first or total: $ S\!I_{i}^C$ and $T\!S\!I_{i}^C$) associated to the membership function of a region $C$ defined over the model output space $\mathcal{C}$. 
\begin{align*}
& S\!I_{i}^C=\frac{\var\left[\esp \left[\mathbb{1}_C(Y) | X_i \right] \right]}{\var\left[\mathbb{1}_C(Y)  \right]}\\
& T\!S\!I_{i}^C=1-\frac{\var\left[\esp \left[ \mathbb{1}_C(Y) | {X}_{\sim i}\right]\right]}{\var\left[\mathbb{1}_C(Y) \right]} 
\end{align*}

In that case, iRSA aims at finding for each input $X_i$ a partition  $(C_i^*,\overline{C_i^*})$ so that $C_i^*$ maximizes either the first or the total region-based Sobol' index of input $X_i$.

By definition, the partition  $(C_i^*,\overline{C_i^*})$ defines the two behaviors whose transition from the one to the other is the most influenced by the model input $X_i$. $C_i^*$ (or equivalently $\overline{C_i^*}$) can also be seen as the region whose occurrence is best explained by the variations of this input.


The optimization problems can be written :
\begin{align*}
&C_i^*=\argmax_{C \in \mathcal{C} }   S\!I_i^C  &( S\!I_i^C \text{based iRSA)} \\
&C_i^*=\argmax_{C \in \mathcal{C} }  \left( T\!S\!I_i^C\right)  &( T\!S\!I_i^C \text{based iRSA)}
\end{align*}

\subsubsection{Indices based on cluster membership differences}
A extension of these indices was also proposed in \cite{roux2021cluster} by considering Sobol' indices on differences of membership functions. 
\begin{align*}
& S\!I_{i}^{CC'}=\frac{\var\left[\esp \left[\mathbb{1}_C(Y)-\mathbb{1}_{C'}(Y) | X_j \right] \right]}{\var\left[\mathbb{1}_C(Y)- \mathbb{1}_{C'}(Y) \right]}\\
& T\!S\!I_{i}^{CC'}=1-\frac{\var\left[\esp \left[ \mathbb{1}_C(Y)- \mathbb{1}_{C'}(Y) | {X}_{\sim i}\right]\right]}{\var\left[\mathbb{1}_C(Y)-\mathbb{1}_{C'}(Y) \right]} 
\end{align*}

In that case, iRSA aims at finding for each input $X_i$ a partition $(C_i^*,C_i'^*,\overline{C_i^* \cup C_i'^*})$ so that that $(C_i^*,C_i'^*)$ maximizes either the first or the total region-based Sobol' index of input $X_i$ based on membership function differences. The partition  $(C_i^*,C_i'^*,\overline{C_i^* \cup C_i'^*})$  defines the two behaviors $(C_i^*,C_i'^*)$ giving the main direction of variation influenced by the model input $X_i$.

The optimization problems can be written :
\begin{align*}
& (C^{1*}_{i},C^{2*}_{i})=\argmax_{C,C' \in \mathcal{C}, C\cap C'=\emptyset }  \left( S\!I_i^{CC'}\right)  & (S\!I_i^{CC'}\text{based iRSA)} \\
&(C^{1*}_i,C^{2*}_i)=\argmax_{C,C' \in \mathcal{C}, C\cap C'=\emptyset }  \left( T\!S\!I_i^{CC'}\right) & ( T\!S\!I_i^{CC'} \text{based iRSA)} 
\end{align*}

\subsection{Generic optimization algorithms}
The previous criteria can be optimized analytically only for simple models (an example is presented in Section \ref{sec_res_1d}). In a general setting, numerical approaches must be considered. For such numerical algorithms, the output space $\mathcal{C}$ is the set of all simulated points. Regions of $\mathcal{C}$ can thus  be seen as clusters of simulated points and iRSA as a clustering algorithm that uses sensitivity-based criteria.

All proposed algorithms require that a classical sensitivity workflow compatible with scalar outputs can be implemented on the model of interest and, more precisely, that a design of experiment $\boldsymbol{X}$ is available for the sensitivity analysis. We note $\boldsymbol{Y}$ the matrix of multivariate outputs obtained when applying the model on the design matrix $\boldsymbol{X}$. We suppose that applying the model on $\boldsymbol{X}$ allows to compute Sobol' type sensitivity indices on scalar outputs (for example, $\boldsymbol{X}$ is defined using a Jansen pick and freeze method that allows the computation of Sobol' indices using the Jansen estimation formula \cite{jansen1999analysis}).

In the following, we propose a set of algorithms implementing iRSA. Their shared principle is to perform a pre-clustering of  $\mathcal{C}$ into $K_Y$ clusters noted $(C_1,...,C_{K_Y})$ in order to reduce the computational cost. 
This pre-clustering is performed using a classical distance-based approach (denoted as \textit{$ClustFun^Y$} in the different algorithms), typically a K-Means algorithm or a hierarchical clustering method, but other methods can be considered to better suit a given application.

\subsubsection{Algorithm for indices based on single cluster membership}
We present a generic algorithm working for any region-based sensitivity criterion $Crit$ provided that it is defined using single membership functions. The most simple criteria are $S\!I_i^{C}$ or $T\!S\!I_i^{C}$ as presented in the previous sections. 

The algorithm has several parameters: the index $i$ of the input under study, the design matrix $\boldsymbol{X}$ and the associated output matrix $\boldsymbol{Y}$, the criterion on single membership function $Crit$, the pre-clustering method $ClustFun^Y$, the number $K_Y$ of clusters for the pre-clustering step and a size parameter $\gamma$ that can be used to avoid partition made of too small regions.
The optimization of $Crit$ is performed by enumerating all partitions of $(C_1,..,C_{K_Y})$ into two non-empty sets. The total number of such  2-partitions is the Stirling number  $\mathscr{S}_2^{K_Y}=2^{{K_Y}-1}-1$ which drives the complexity of the algorithm. The algorithm returns the best partition under the size constraint as well as the associated sensitivity score.
We denote as iRSA\_SM (SM for for Single Membership) this algorithm:

\begin{algorithm}
\caption{iRSA\_SM$(i,\boldsymbol{X},\boldsymbol{Y},Crit,ClustFun^Y,K_Y,\gamma)$ }
\label{algdSI}
\begin{algorithmic}
\STATE{Apply $ClustFun^Y(K_Y)$ on $\boldsymbol{Y}$ to get $K_Y$ clusters $(C_1,..,C_{K_Y})$}
\FOR{all 2-partitions $(P_k, \overline{P_k})_{k \in [1:\mathscr{S}_2^{K_Y}]}$ of $(C_1,..,,C_{K_Y})$}
\STATE{Compute $\gamma_k=\min\left(Card(P_k),Card(\overline{P_k})\right)$}
\STATE{Compute $Crit(i,\boldsymbol{X},\boldsymbol{Y},P_k)$}
\ENDFOR
\STATE{Get $k^*=\argmax
_{k\in [1:\mathscr{S}_2^{K_Y}], \gamma_k \geq \gamma} Crit(i,\boldsymbol{X},\boldsymbol{Y},P_k)$ }
\RETURN    $(P_{k^*},\overline{P_{k^*}})$ and  $Crit(i,\boldsymbol{X},\boldsymbol{Y},P_{k^*})$
\end{algorithmic}

\end{algorithm}

Note that by storing all sensitivity indices for each partition in the loop, it is possible to process all inputs at once during the post-processing step, thus avoiding to run the loop again for different inputs.





\subsubsection{Algorithm for indices based on cluster membership differences}
In the same spirit as the previous algorithm, the optimization of a sensitivity criterion based on membership function differences such as $S\!I_i^{CC'}$ or $T\!S\!I_i^{CC'}$ (and denoted generically as \textit{Crit} in the detail of the algorithm), is performed by enumerating all partitions. Due to the structure of these indices, it is necessary to enumerate all partitions of $(C_1,..,C_{K_Y})$ into three non-empty sets. The total number of such  3-partitions is the Stirling number $\mathscr{S}_3^{K_Y}=\frac{1}{6}\left( 3^{K_Y} - 3\cdot 2^{K_Y} +3  \right)$.
We denote as iRSA\_DM (DM for Difference of Membership) this algorithm.

\begin{algorithm}
\caption{iRSA\_DM$(i,\boldsymbol{X},\boldsymbol{Y},Crit,ClustFun^Y,K_Y,\gamma)$ }
\label{algdSI}
\begin{algorithmic}
\STATE{Apply $ClustFun^Y(K_Y)$ on $\boldsymbol{Y}$ to get $K_Y$ clusters $(C_1,..,C_{K_Y})$}
\FOR{all 3-partitions $\left(Q_1^k,Q_2^k,Q_3^k\right)_{k \in [1:\mathscr{S}_3^{K_Y}]}$ of $(C_1,..,,C_{K_Y})$}
\STATE{Compute $\gamma_k=\min\left(Card(Q_1^k),Card(Q_2^k),Card(Q_3^k)\right)$}
\STATE{Compute for $(p,q)\in I=\{(1,2),(1,3),(2,3)\}$,   $Crit(i,\boldsymbol{X},\boldsymbol{Y},Q_p^k ,Q_q^k)$}
\ENDFOR
\STATE{Get $(k^*,p^*,q^*)=\argmax
_{k\in [1:\mathscr{S}_3^{K_Y}],(p,q)\in I, \gamma_k \geq \gamma} Crit(i,\boldsymbol{X},\boldsymbol{Y},Q_p^k ,Q_q^k)$ }

\RETURN    $\left(Q_1^{k^*},Q_2^{k^*},Q_3^{k^*}\right)$ and  $Crit(i,\boldsymbol{X},\boldsymbol{Y},Q_{p^*}^{k*}, Q_{q*}^{k*})$
\end{algorithmic}
\end{algorithm}

\subsection{Improved algorithm for first order indices based on single cluster membership}
The previous algorithms are computationally intensive because of the exhaustive search over all partitions (2-partitions when optimizing indices of cluster membership, 3-partitions when  optimizing indices of cluster membership differences). The associated costs are driven by the loop over these sets and is exponential with respect to the number of clusters of the pre-clustering set $(O(2^ {K_Y}),O(3^{K_Y}))$. In practice, it becomes intractable to consider ${K_Y}$ larger than $15-20$, which limits the exploration to large-sized clusters. This limit motivates the proposition of an improved algorithm that would take advantage of the properties of the cost functions (which were not taken into account in the different algorithms presented in the previous section). As the study then becomes dependent on the chosen criterion, we present here an improvement of the optimization based on $S\!I_i^{C}$ (first order Sobol' index on single cluster membership).

We start by giving a property that helps building partitions from a set of elementary clusters $(C_k)$ of the output space.

\subsubsection{An aggregating property of clusters }
We consider a model input $X$  and suppose that it has a uniform distribution over an interval $I$. Let $ (C_k)_{k \in [1,K]}$ be a set of non overlapping clusters of the output space. 
We consider the distribution of $X$ given $Y \in C_k$ and denote as $h^{C_k}$ the associated histograms obtained when discretizing $I$ into $n_x$ bins. Let $N$ be the total number of simulated points. Then the discrete approximation of criterion ${S\!I}^{C_k}$, noted $\widetilde{S\!I}^{C_k}$, can be written using  $h^{C_k}$:    
$$\widetilde{S\!I}^{C_k} = {n_x \over \displaystyle \sum_{j=1}^{n_x} h_j^{C_k} (N- \sum_{j=1}^{n_x} h_j^{C_k})} \sum_{i=1}^{n_x} (h_i^{C_k} - {1 \over n_x} \displaystyle \sum_{j=1}^{n_x} h_j^{C_k})^2$$

\begin{proposition} Let’s consider two  elementary clusters $C_k$ and $C_{k’}$ with perfectly correlated histograms $h^{C_k}$ and $h^{C_k'}=\theta h^{C_k}$.
Let’s denote $(C^*,\overline{C^*})$ the partition of the output space built from the $(C_k)$ that maximizes $\widetilde{S\!I}^{C}$ (discretized first order index based on single cluster membership). Then  $C_k$ and $C_{k’}$ belong both to $C^*$ or to $\overline{C^*}$.
\label{prop_hist}
\end{proposition}

\begin{proof}
See Appendix B.
\end{proof}
This property states that elementary clusters with perfectly correlated histograms can not be in different sets of the optimal partition. 
\subsubsection{Algorithm}
We derived an efficient algorithm named iRSA\_SM' by taking advantage of Proposition  \ref{prop_hist}. The principle of iRSA\_SM' is to keep the overall structure of Algorithm iRSA\_SM and to improve the pre-processing step by using Proposition  \ref{prop_hist}. More precisely, Proposition  \ref{prop_hist} is used to propose a merging  of elementary clusters based on the correlation of their histograms. Indeed, if the histograms are perfectly correlated, then from  Proposition  \ref{prop_hist}, the associated elementary clusters can be merged.

The pre-clustering step of Algorithm iRSA\_SM' consists in two steps. Firstly, as  for iRSA\_SM, a  pre-clustering of $\boldsymbol{Y}$ using \textit{$ClustFun^Y$} is performed to get $K_Y$ clusters, but this time with a larger $K_Y$ (500-1000). Secondly\textcolor{green}{,} a clustering of these $K_Y$ clusters into $K_H$ meta-clusters is done using a distance-based method denoted \textit{$ClustFun^H$}. \textit{$ClustFun^H$} can be the same as \textit{$ClustFun^Y$} but is based on a different distance: \textit{$ClustFun^H$} works with an histogram-based correlation distance (e.g. $d(h,h')=1-Cor(h,h')$) and is applied on the elementary histograms associated to the conditional distributions of $X_i$ given $Y \in C_k$. The last step of the algorithm, the exhaustive search for the best partition over all possible merging of elementary clusters, is kept unchanged.

Compared to iRSA\_SM, Algorithm iRSA\_SM' has three additional parameters: the number $n_X$ of bins for the histograms computations, the generic method \textit{$ClustFun^H$} and the number of meta-clusters $K_H$.

\begin{algorithm}
\caption{iRSA\_SM'$(i,\boldsymbol{X},\boldsymbol{Y},Crit,ClustFun^Y,K_Y,n_x,ClustFun^H,K_H,\gamma)$ }
\label{algdSI}
\begin{algorithmic}
\STATE{Apply $ClustFun^Y(K_Y)$ on $\boldsymbol{Y}$ to get $K_Y$ clusters $(C_1,..,C_{K_Y})$, with $K_Y$ large }
\STATE{Compute all elementary histograms $(h^{C_k})_{k=1..K_Y}$ for input $X_i$ discretized with $n_x$ bins}
\STATE{Apply $ClustFun^H(K_H)$ on $(h^{C_1},..,h^{C_{K_Y}})$ to get  $K_H$ meta-clusters $(\hat{C}_1,..,\hat{C}_{K_H})$ based on histogram correlation}
\FOR{all 2-partitions $(P_k, \overline{P_k})_{k \in [1:\mathscr{S}_2^{K_H}]}$ of $(\hat{C}_1,..,\hat{C}_{K_H})$}
\STATE{Compute $\gamma_k=\min\left(Card(P_k),Card(\overline{P_k})\right)$}
\STATE{Compute $Crit(i,\boldsymbol{X},\boldsymbol{Y},P^k)$}
\ENDFOR
\STATE{Get $k^*=\argmax
_{k\in [1:\mathscr{S}_2^{K_H}], \gamma_k \geq \gamma} Crit(i,\boldsymbol{X},\boldsymbol{Y},P^k)$ }
\RETURN $(P_{k^*},\overline{P_{k^*}})$ and  $Crit(i,\boldsymbol{X},\boldsymbol{Y},P^{k^*})$
\end{algorithmic}

\end{algorithm}

Note that this time, the algorithm is strongly dedicated to the optimization related to a single input $X_i$ due to the histogram computation step . It is thus only possible to re-use the first pre-clustering when performing the optimization for another input.

\section{Application of the method on various models}
\label{sec_results}

\subsection{Test model with 1d outputs: analytical resolution}
In this section, we present the analytical resolution of the optimization problem on a simple 1D model in order to highlight the main issues of the iRSA   approach.
\label{sec_res_1d}
\subsubsection{Model and input distributions}
We consider the model with two input factors $X_1$ and $X_2$ defined by : $$Y=Sign(X_1) \cdot |X_2|$$
The model is studied over the domain defined by $X_1$ and $X_2$ having independent uniform distributions on $[-1,1]$. $Y$ thus also takes its values within $[-1,1]$.

\subsubsection{Optimization problem}
 We consider the problem of finding for $X_1$ (resp. $X_2$) the partition $(C,\bar{C})$ of the one-dimensional output space $[-1,1]$ the most influenced by the variations of $X_1$ (resp. $X_2$).
$$
C_i^*=\argmax_{C\subset [-1,1]} SI_i^C, \   i=1,2
$$

We limit the search to partitions that can be parameterized with a single cutting value $y_c$ or with two cutting values $y_{c_1}$ and $y_{c_2}$. In the first case, it amounts to consider partitions of type 'A-B', and in the second case to partitions of type 'A-B-A' of the 1-dimensional output space.

\subsubsection{Analytical expressions of the region-based indices}
The simplicity of this model allows to derive analytical expressions of the sensitivity indices $(SI_i^C)_{i=1,2}$ for the different types of partition using one  cutting values $y_c$ or two cutting values $ y_{c1}$ and $y_{c2}$ (see Appendix A for details on how to obtain these expressions).

\begin{table}
\begin{tabular}{|l|c|c|}
     \hline
     & $S\!I_i^C(y_c)$ &$S\!I_i^C(y_{c_1},y_{c_2})$  \\
     \hline
    $X_1$ & $ {(1-|y_c|)^2 \over 1- y_c^2}$ & \multicolumn{1}{|l|}{ ${(|y_{c_2}|-|y_{c_1}|)^2 \over (y_{c_2}-y_{c_1}) (2-y_{c_2}+y_{c_1})} $ }\\
    $X_2$ & ${|y_c| \over 1+|y_c|}$ & $\begin{cases}
 {1-y_{c_2}+y_{c1} \over 2-y_{c_2}+y_{c_1} }  & \mbox{ if } y_{c_2} \leq 0 \mbox{ or } y_{c_1} \geq 0 \\
 {1-y_{c_2}+y_{c_1} \over 2-y_{c_2}+y_{c_1} } + 2 {\min(|y_{c_1}|,y_{c_2})\over (y_{c_2}-y_{c_1})(2-y_{c_2}+y_{c_1})} & \mbox{ if } y_{c_1} \leq 0 \leq y_{c_2} 
\end{cases}$\\
\hline
\end{tabular}
\caption{Analytical expressions of the region-based sensitivity indices for model  $Y=Sign(X_1) \cdot |X_2|$   \label{tab_ex1d}}
\end{table}

\subsubsection{Optimization}
 The expressions presented in Table \ref{tab_ex1d} allow to derive the optimal cutting values that maximize the sensitivity indices for the two parameters in the different cases.
\paragraph{Single cutting value $y_c$}
In this case, as can be seen in Figure \ref{fig_S1_model1d}, it is easy to show that:
\begin{itemize}
\item $S\!I_1^C$ is maximum for $y_c=0$. For this partition, we have $S\!I_1^C=1$ and $S\!I_2^C=0$.
\item $S\!I_2^C$ is maximum when $y_c$ reach the boundaries of the domain ($y_c=-1$ or $y_c=1$). In this case,  $S\!I_2^C ={1 \over 2}$. For these partitions, we have $S\!I_1^C=0$.
\end{itemize}
We conclude that the single cutting value case allows to find an optimal partition $(C_1^*,\overline{C_1^*})=([-1,0],[0,1])$ for input $X_1$ independently of the hypothesis on the number of cutting values (as  $S\!I_1^C$ maximum value is $1$). However, the case of a single cutting value does not allow to find a partition fully explained by the variations of $X_2$.

\begin{figure}[!ht]
 \begin{center}
 \includegraphics[width=1.0 \textwidth]{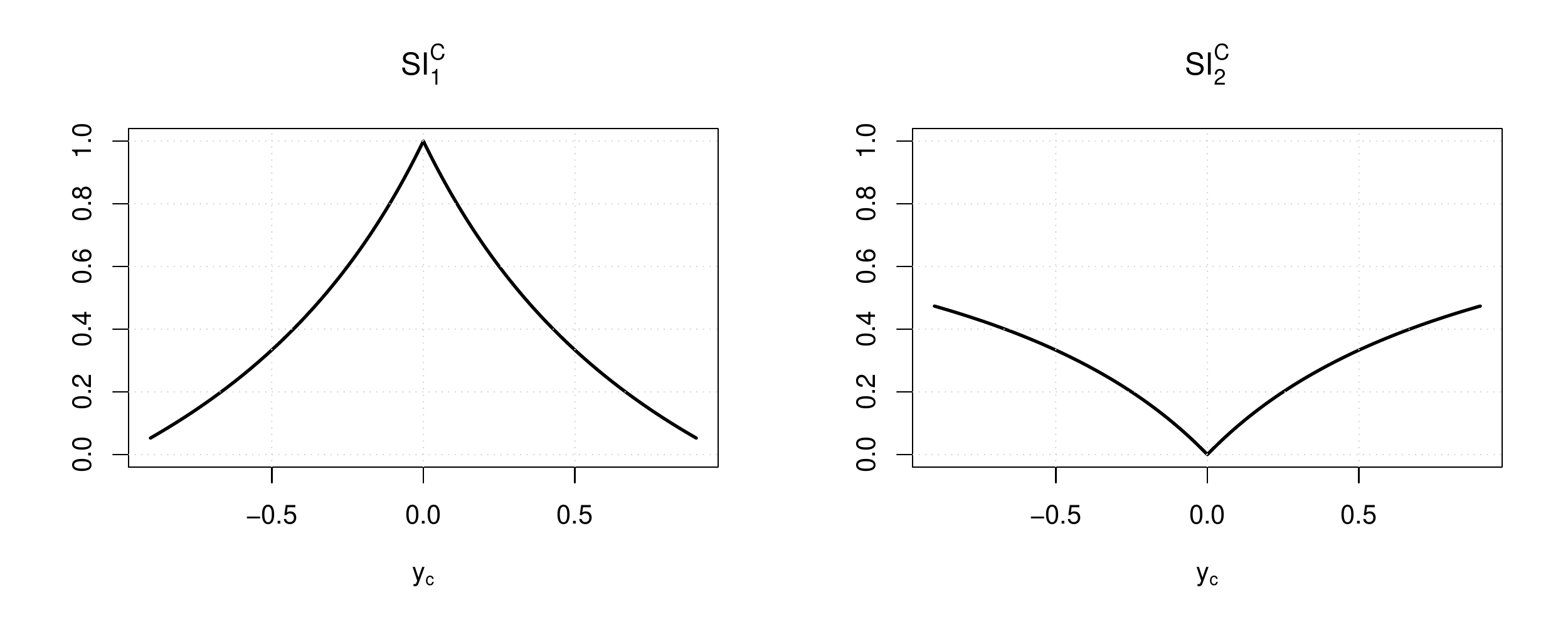}
 \end{center}
 \caption{First order indices $S\!I_i^C$ as a function of a single cutting value $y_c$ for the two inputs of model $Y=Sign(X_1) \cdot |X_2|$
 \label{fig_S1_model1d}}
 \end{figure} 
 
\paragraph{Two cutting values}
In this case, as can be seen in Figure \ref{fig_S2_model1d}, we can show that:
\begin{itemize}
\item $S\!I_1^C$ is maximum for $\displaystyle (y_{c_1},y_{c_2})=(-1,0)$ and $(y_{c_1},y_{c_2})=(0,1)$. This values are degenerated situations and correspond to the single cutting value case.
\item $S\!I_2^C$ is maximum when $y_{c_1}=-y_{c_2}$. For these partitions, we have $S\!I_2^C=1$ and $S\!I_2^C=0$.
\end{itemize}
 Considering the case of two cutting values allows to find a set of partitions with $S\!I_2^C=1$. These partitions can be written  $(C_2^*,\overline{C_2^*})=\left([-\delta,\delta],[-1,-\delta]\cup [\delta,1]\right)$, with $\delta \in ]0,1[$.

\begin{figure}[!ht]
 \begin{center}
 \includegraphics[width=1.0 \textwidth]{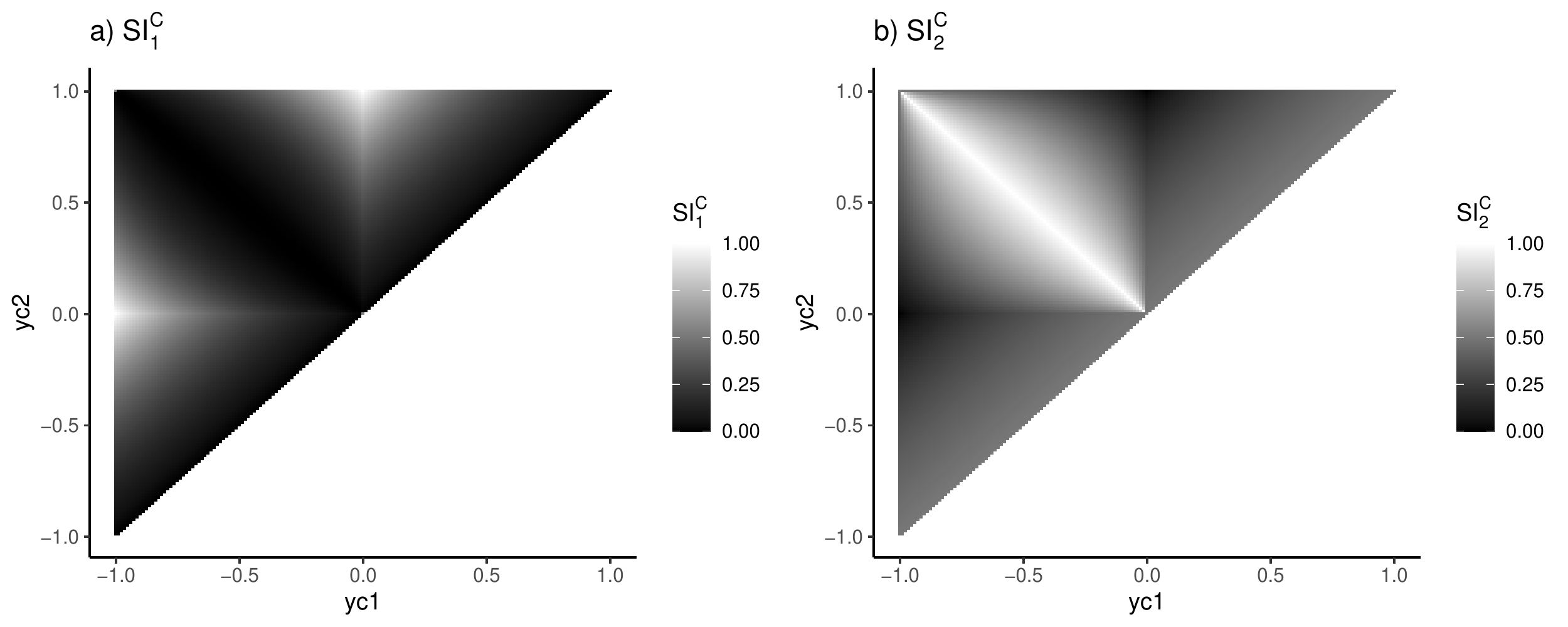}
 \end{center}
 \caption{First order indices $S\!I_i^C$ as a function of two cutting values $y_{c_1}$ and $y_{c_2}$ for the two inputs of model  $Y=Sign(X_1) \cdot |X_2|$
 \label{fig_S2_model1d}}
 \end{figure}

\subsubsection{Conclusion of the example}

Using the analytical expressions of the sensitivity indices given previously, the iRSA sensitivity approach allows to conclude that $X_1$ has the effect of leading negative output to positive outputs, and that $X_2$ has the effect of leading the output from the center of the domain toward its boundaries. 

This example also  highlights some issues about iRSA:
\begin{itemize}
\item The optimal partition is associated to a model input and may be different depending on the input factor considered.
\item Optimal partitions can be found even if the space of partitions is not completely explored. In this example, exploring partitions defined by one or two cutting values was sufficient to find global optima.
\item We can have $S\!I_i^{C^*}>S\!I_i$ ($S\!I_i$ being the standard first order Sobol' index of input factor $X_i$). In this example, we have $S\!I_2^{C^*}=1$ while $S\!I_2=0$: the information provided by iRSA is different from the one provided by  a standard variance-based sensitivity analysis.
\item The optimal partition may not be unique. In this example, we found a set of optimal partitions for $X_2$.

\end{itemize}
\subsection{Test model with 2D outputs: illustration of the different SI-based criteria}
\label{sec_res_2d}

\subsubsection{Model and input distributions}
We consider the model having 2D outputs $(Y_1,Y_2)$ and four inputs $(X_1,X_2,X_3,X_4)$. Basically, the model draws a point in the plane at a distance related to $X_4$ and with an angle related to $X_3$. This angle is defined from one of three given centers, with a choice of center related to $X_1$ and $X_2$ in interaction. More precisely, the model definition is the following:

$$
\begin{array}{l}
Y_1 = c_1 + 0.4 \cdot X_4^3 \cdot \cos( 2 \pi X_3)  \\
 Y_2  = c_2 + 0.4 \cdot X_4^3 \cdot \sin( 2 \pi X_3) 
\end{array}
\text{where } (c_1,c_2)= \begin{cases}
    (0.5,0.25), & \text{if } X_1<0.5   \\
    (0.25,0.75), & \text{if } X_1 \geq 0.5  \text{ and } X_2 < 0.5  \\
    (0.75,0.75), & \text{if } X_1 \geq 0.5  \text{ and } X_2 \geq 0.5 
    \end{cases} 
$$
All inputs have independent uniform distribution in $[0,1]$.
The sensitivity indices were computed using the SobolSalt function of the R package sensitivity \cite{Pujol2017} with a sample size of $n=1000$, producing $10000$ output values.

\subsubsection{Optimization setting}
 We illustrate the three numerical algorithms of Section \ref{sec_method} on this 2D toy model. More precisely, we apply for each input $(X_i)_{i=1..4}$ Algorithm iRSA\_SM and its improved version iRSA\_SM' in order to find the clustering maximizing $S\!I_i^C$. We also apply Algorithm iRSA\_DM in order to find the clustering maximizing $S\!I_i^{CC'}$.
 For algorithms iRSA\_SM and iRSA\_DM, we use a pre-clustering of the simulated results in $K_Y=10$ elementary clusters using a fuzzy k-means clustering approach (which is well adapted for small $K_Y$). This pre-clustering is presented in Figure \ref{fig_K10}. For Algorithm  iRSA\_SM', the first pre-clustering was performed in $K_Y=1000$ elementary clusters. The histogram correlation method was then applied to obtain $K_H=10$ clusters before the application of the exhaustive search step with a size constraint parameter $\gamma=10\%$. A hierarchical clustering approach (more adapted for large $K_Y$) was used for $ClustFun^Y$ (with an euclidean distance on Y) and $ClustFun^H$ (with a correlation distance on histograms). 


\begin{figure}[!ht]
 \begin{center}
\fbox{\includegraphics[width=0.4 \textwidth]{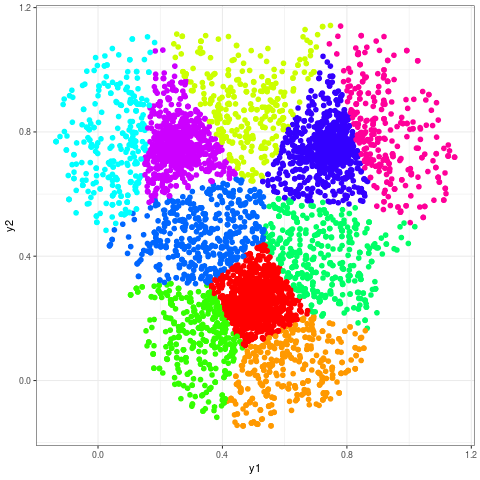}}
 \end{center}
\caption{Clustering of the 2D test model outputs into $K_Y=10$ clusters 
 \label{fig_K10}}
 \end{figure}
 
\subsubsection{Method for displaying clusters}
\label{sec_visu2d}
 
For Algorithm iRSA\_SM and iRSA\_DM, we present the clustering into two or three clusters and the associated region-based sensitivity indices. Algorithm iRSA\_SM' must be handled differently, as the algorithm does not guaranty clear boundaries between clusters. To handle this issue, we propose a different display of its results: we associate to each pixel a continuous color level (noted $mmv$, for mean membership value) linked to the proportion of output points belonging to each of the two clusters. Using this definition, pixels with output points from cluster A only (resp. B only) will be displayed in blue (resp. red), but pixels with output points from the two clusters will have a clearly different color. This display allows to quickly spot zones where the clusters boundaries are blurred and where interpreting the regions can be difficult.

  \subsubsection{Clustering based on $S\!I^C$ with Algorithm iRSA\_SM}
 \begin{figure}[!ht]
 \begin{center}
  \fbox{\includegraphics[width=1 \textwidth]{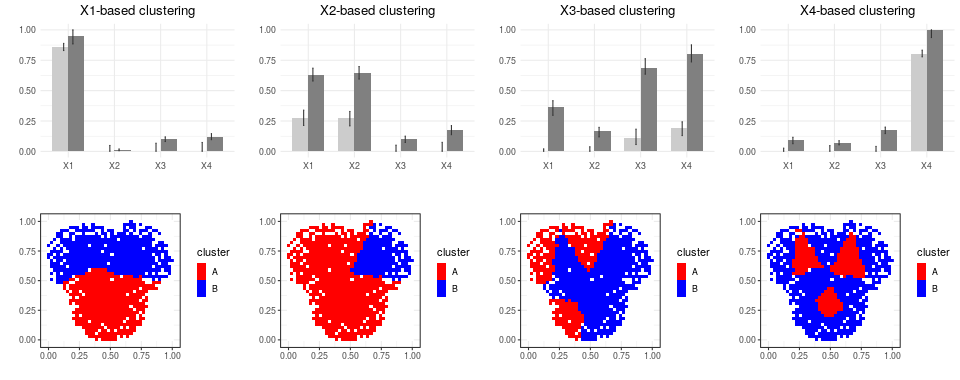}}
 \end{center}
 \caption{Results obtained with the $S\!I^C$ criterion and Algorithm iRSA\_SM  on the 2D toy model. First line: region-based sensitivity indices of the best clustering for each input. Second line: visualization of the clusters for each input.
 \label{fig_si}}
 \end{figure} 
 We present in Figure \ref{fig_si} the result obtained with Algorithm iRSA\_SM. For inputs $X_1$ and $X_4$, the clustering found is associated to very high region-based indices ($S\!I^{C^*}_1=0.86$, $S\!I^{C^*}_4=0.80$). These scores reveal a strong effect of these inputs in the occurrence of the obtained clusters. For $X_1$, two large clusters are obtained with a boundary approximately equal to the line $y_2=0.5$. For $X_4$, the clustering distinguishes a region made of three zones (each one close to one of the three centers) from a region made of points far from all centers. Note that the pre-clustering does not allow to get perfect circles around the centers. Nevertheless, the optimized region-based sensitivity indices are very high. For input $X_3$, the sensitivity score associated to the clustering is low and prevents any further interpretation. For $X_2$, we obtain a relatively low sensitivity score ($S\!I^{C^*}_2=0.27$) but a clear structure of clusters: the optimal partition is made from points closest to top right center. This was expected, as $X_2$ drives the choice of the two centers of the top. But as this choice is made conditionally to $X_1$, there is a strong interaction effect with $X_1$ that prevents from getting first order indices close to one.
\subsubsection{{Clustering based on $S\!I^{CC'}$ with Algorithm iRSA\_DM}}
The optimization of criterion $S\!I^{CC'}$ tries to find two regions whose transition from the one to the other is the most influenced by the variations of an input. By construction it leads to higher scores compared to optimizing $S\!I^{C}$ since $S\!I^{C}$ is a particular case of  $S\!I^{CC'}$ with $C'=\emptyset$.
When looking at Figure \ref{fig_dclustSI}, we can verify this property for $X_1$ and $X_4$. We can also see that the result obtained on $X_3$ still prevents any interpretation. For $X_4$, the indices are still high, but the interpretation is not made easier. On the other hand for $X_2$, the algorithm found an improved value of the first order index ( $S\!I^{CC'^*}_2=0.41$ ) associated with a value of its total index  $T\!S\!I^{CC'^*}=0.97$  that only $X_2$ reaches. This  means that $X_2$ is the key model input that explains the transitions between the red cluster (points closest to the top right center) and the blue one (points closest to the top right center). Note again that it is not surprising to still get a strong interaction with $X_1$ given the model definition. This example shows the interest of considering the $S\!I^{CC'}$ criterion.
 \begin{figure}[!ht]
 \begin{center}
 \fbox{\includegraphics[width=1 \textwidth]{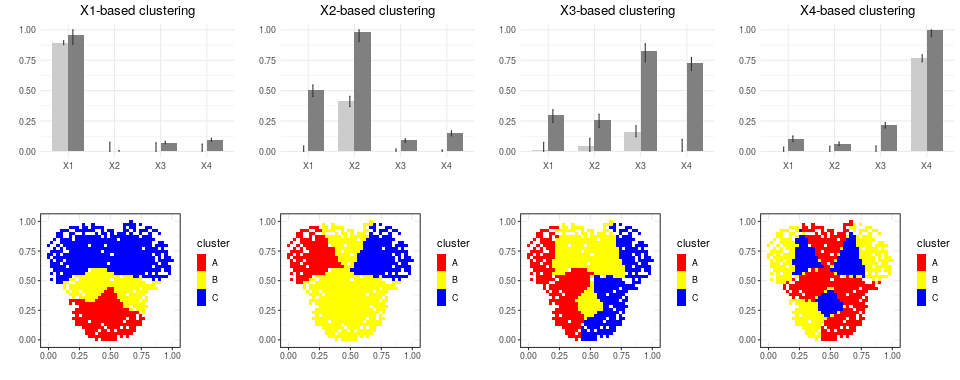}}
 \end{center}
 \caption{Results obtained with the $S\!I^{CC'}$ criterion and Algorithm iRSA\_DM on the 2D toy model. First line: region-based sensitivity indices of the best clustering for each input. Second line: visualization of the clusters  (red and blue with neutral class in yellow) for each input.
 \label{fig_dclustSI}}
 \end{figure} 


\subsubsection{Clustering based on $S\!I^C$ with Algorithm iRSA\_SM'}
The results of the application of Algorithm iRSA\_SM' are presented in Figure \ref{fig_si_imp}. The algorithm allows to generate clusters with more flexibility in terms of boundaries and  spatial resolution than Algorithm iRSA\_M. For $X_1$ and $X_4$, the sensitivity indices were already well optimized using the first algorithm. However, this score is significantly improved for $X_4$ using Algorithm iRSA\_SM'. We can also note that the three red zones identified for $X_4$ have the clear expected circular shape, which was not possible to obtain with the first algorithm. For $X_2$, the scores are still low ($S\!I^{C^*}_3=0.30$) but in this case, it is due to the definition of the model, not to a limit of the algorithm. A clear improvement was obtained for $X_3$, for which the sensitivity score associated to the best clustering is strongly improved ($S\!I^{C^*}_3=0.57$ against $S\!I^{C^*}_3=0.11$ for Algorithm iRSA\_SM) and where the expected  angular pattern starts to appear. The three red zones correspond approximately to the same angular section around their closest centers, which was expected given the model definition.
\begin{figure}[!ht]
 \begin{center}
 \fbox{\includegraphics[width=1 \textwidth]{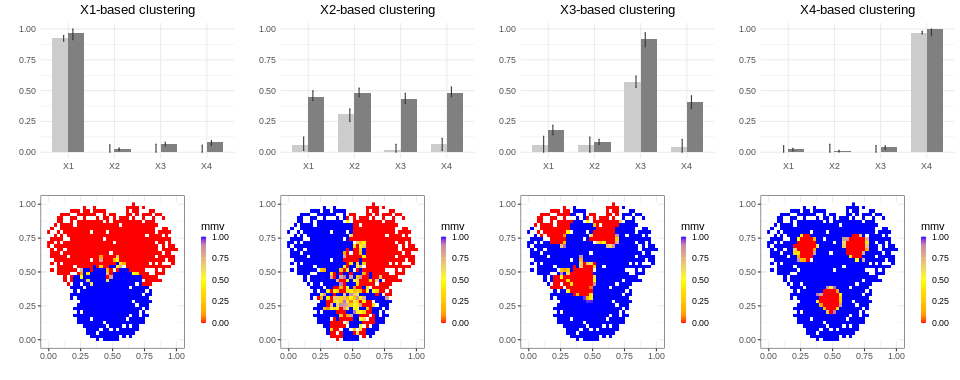}}
 \end{center}
 \caption{Results obtained with the $S\!I^C$ criterion and Algorithm iRSA\_SM'  on the 2D toy model. First line: region-based sensitivity indices of the best clustering for each input. Second line: visualization of the clusters for each input using mean membership maps.
 \label{fig_si_imp}}
 \end{figure} 

To conclude, we have shown on this example that Algorithm iRSA\_SM' managed to automatically identify regions of the output space that maximize first-order region-based Sobol' indices. These regions were quite precisely delimited and are consistent with what was expected, given the model definition.

\subsection{Test model producing time series: application of Algorithm iRSA\_SM'}
 \label{sec_res_nd}
 
\subsubsection{Model and input distributions}
We apply the iRSA approach on the environmental model CANTIS \cite{garnier2001}, which simulates the decomposition of organic biomass in crop residues over time. We use the same model setting as in \cite{roux2021cluster} and invite the reader to refer to this article to get more details. The model setting consists in ten uncertain model inputs and the analysis of one time-varying output: the zymogenous microbial biomass, here-after referred to as $\ZYB$. The sensitivity indices are  computed using the SobolSalt function of the R package sensitivity \cite{Pujol2017} with a sample size $n=2500$ producing $55000$ output curves of $\ZYB$.
\subsubsection{Optimization setting}
Algorithm iRSA\_M' is applied for each input $(X_i)_{i=1..10}$ in order to find the clustering maximizing $S\!I_i^C$. A hierarchical clustering approach (well adapted to large $K_Y$) is used for $ClustFun^Y$ (with an euclidean distance on Y) and $ClustFun^H$ (with a correlation distance on histograms). The first pre-clustering is performed with $K_Y=500$ clusters and the second with $K_H=10$ clusters before the application of the exhaustive step with a size constraint of $\gamma=10\%$. 

\subsubsection{Method for displaying clusters}
In this case, iRSA produces two sets of curves for each input factor. As in the 2D example, we represent the result using mean membership maps (see Section \ref{sec_visu2d}). We also plot for each cluster the quantiles of the distributions obtained at each time step. In order to quickly assess whether at a given time, the distribution of values of the two clusters differ a lot or not, we add an horizontal colorbar corresponding to the Kolmogorov–Smirnov test statistic (maximum value of the difference between cumulated distributions).
\label{sec_visu_nd}
\subsubsection{Results}
We present in Table \ref{tabl_sixi} the sensitivity indices $S\!I_1^{C^*}$ obtained when applying Algorithm iRSA\_SM' for each input of the CANTIS model in the considered uncertainty setting. Only $X_8$ and $X_9$ exhibits large values that deserve further analysis. For these two inputs, we present in Figure \ref{fig_x8x9} the clusters using the visualization approach detailed in Section  \ref{sec_visu_nd}, which is necessary to handle large-sized data set (55000 clustered curves).

For input $X_8$, we obtain $SI_8^{C*}=0.843$ for the optimized clustering. Looking at Kolmogorov–Smirnov test statistic in Figure \ref{fig_x8x9}, we see that the two clusters differ the most at the beginning of the simulation. 
More precisely, when looking at the mean membership map, we see that the red cluster is composed mostly of curves that start to increase and the blue one  of curves that start to decrease. The two clusters then overlap rapidly.
The application of the method let us conclude that $X_8$ drives the occurrence of biomass dynamics that start by increasing (or equivalently that start by decreasing).

For input $X_9$, a very high score $SI_9^{C*}=0.868$ is obtained using Algorithm iRSA\_SM'. For this input, when looking  Kolmogorov–Smirnov test statistic, it appears that what best  distinguishes the two clusters is the level of $ZY\!B$ values over the entire simulated time period, and particularly the final value. 
We can conclude that $X_9$ drives the occurrence of dynamics with highest (or lowest) biomass levels all along the considered time scale.
\begin{table}
\begin{tabular}{|c|c c c c c c c c c c|} 
 \hline
 input & $X_1$& $X_2$& $X_3$& $X_4$& $X_5$& $X_6$& $X_7$& \boldsymbol{$X_8$}& \boldsymbol{$X_9$}& $X_{10}$ \\ 
 \hline
 $SI_1^{C^*}$  & $0.056$& $0.054$& $0.057$& $0.057$& $0.042$& $0.049$& $0.038$& \boldsymbol{$0.843$}& \boldsymbol{$0.868$}& $0.031$\\
 \hline
\end{tabular}

\caption{Optimized region-based sensitivity indices obtained for each input $X_i$.\label{tabl_sixi}  }
\end{table}

\begin{figure}[!ht]
 \begin{center}
\fbox{ \includegraphics[width=1.0 \textwidth]{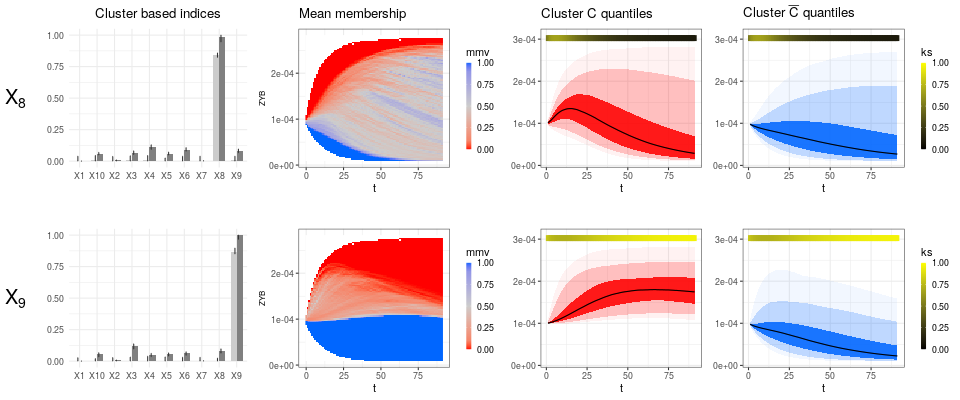}}
 \end{center}
 \caption{Clustering obtained with Algorithm iRSA\_SM' on the CANTIS model for input $X_8$ (first line)  and $X_9$ (second line). First column: region-based sensitivity indices of the best clustering; second column: mean membership maps; third and fourth columns: quantiles (0.05,0.25,0.50,0.75,0.95) of the two clusters of curves at each time. The horizontal colorbar stresses significant differences between the two distributions  $Y(t)|(Y\in C)$ and $Y(t)|(Y\in \bar{C})$. : yellow (resp. black) correspond to a high (resp. low)  value of the Kolmogorov–Smirnov (ks) test statistics applied to the two distributions.
 \label{fig_x8x9}
 }
 \end{figure}



\section{Conclusion}
In this work, we introduced a new sensitivity analysis approach (named iRSA) for exploring models with various output dimensions by proposing a new perspective on the concept of 
Regional Sensitivity Analysis. Instead of a priori defining target model behaviors, we automatically detect the ones whose occurrences are best explained by the variations of the different model inputs. This is a  new way of exploring simulation models. \newline 
The method was illustrated by analytically solving a simple scalar case. Different algorithms have been proposed  to compute numerical solutions in the case of complex models as well as adapted graphical representations to interpret the results. The relevance of the results was checked on a model with 2D outputs. iRSA was also applied on an environmental model producing time series, showing that interpretable information can be produced by the method.
\newline
Future work on iRSA will be focused on i) the extension of the proposed sensitivity criteria in order to focus on groups of inputs and interactions, ii) attempting to derive efficient algorithms for other criteria (as we did for the first order indices), iii) the application of the method on models with other kind of complex outputs, such as spatial ones.




\bibliographystyle{elsarticle-num}
\bibliography{./references}


\section*{Appendix A: region-based indices for model $(Sign(X_1)\cdot |X_2|)$}
We derive for this model the expressions of $S\!I_i^C$ when $X_1$ and $X_2$ have independent uniform distributions over $[-1,1]$ for two parameterizations of an interval $C \subset [-1,1]$  : one using a single cutting value $y_c$ and another using two cutting values $y_{c1}$ and $y_{c2}$.  
\subsection*{A.1 Parameterization with one cutting value $y_c$}
In this case, we study $f(X_1,X_2) = \mathbb{1}_{sign(X_1)|X_2| \leq y_c} $
\subsubsection*{A.1.1 Derivation of $S\!I_1^C(y_c)$}
\paragraph{\textbf{Conditional expectation}} $\esp[f|X_1]= {1 \over 2} \int_{-1}^{1} \mathbb{1}_{sign(x_1)|x_2| \leq y_c} d x_2$:

If $x_1 < 0$: $\esp[f|X_1] = $ 
$\begin{cases}
1 & \mbox{ if } y_c \geq 0 \\
{1 \over 2} (\int_{-1}^{y_c} d x_2 +\int_{-y_c}^1 d x_2) = 1+y_c & \mbox{ if } y_c \leq 0 \\
\end{cases}$

If $x_1 > 0$: $\esp[f|X_1] = $ 
$\begin{cases}
{1 \over 2} \int_{-y_c}^{y_c} d x_2 = y_c & \mbox{ if } y_c \geq 0 \\
0 & \mbox{ if } y_c \leq 0 \\
\end{cases}$
\paragraph{\textbf{Variance of conditional expectation}} $ \var[\esp[f|X_1]]$

$\displaystyle {1 \over 2} \int_{-1}^1 \esp[f|x_1] dx_1 = {1 \over 2} (\int_{-1}^0 \esp[f|x_1] dx_1+ \int_0^1 \esp[f|x_1] dx_1) = {1+y_c \over 2}$

$\displaystyle {1 \over 2} \int_{-1}^1 \esp[f|x_1]^2 dx_1 = {1 \over 2} (\int_{-1}^0 \esp[f|x_1]^2 dx_1+ \int_0^1 \esp[f|x_1]^2 dx_1) = $ 
$\begin{cases} 
{1+y_c^2 \over 2} & \mbox{ if } y_c \geq 0 \\
{(1+y_c)^2 \over 2} & \mbox{ if } y_c \leq 0 \\
\end{cases}$
Therefore $\displaystyle \var[\esp[f|X_1]] = {1 \over 2} \int_{-1}^1 \esp[f|x_1]^2 dx_1 - ({1 \over 2} \int_{-1}^1 \esp[f|x_1] dx_1)^2={(1-|y_c|)^2  \over 4}$

\paragraph{\textbf{Variance}} 
$\var[f]=\esp[f^2]-\esp[f]^2$ with $\esp[f^2]= \esp[f]$ so $\displaystyle \var[f] = {1+y_c \over 2} - ({1+y_c \over 2})^2= {1 - y_c^2 \over 4}$
\paragraph{\textbf{Sobol index} }
 $\displaystyle S\!I_1^C(y_c)= {\var[\esp[f|x_1]] \over \var[f]} = {(1-|y_c|)^2 \over 1- y_c^2}$
\subsubsection*{A.1.1 Derivation of $S\!I_2(y_c)$}

\paragraph{\textbf{Conditional expectation}} $\esp[f|X_2]= {1 \over 2} \int_{-1}^{1} \mathbb{1}_{sign(x_1)|x_2| \leq y_c} d x_1$:

\begin{align}
 \esp[f|X_2] & = {1 \over 2} \int_{-1}^{1} \mathbb{1}_{sign(x_1)|x_2| \leq y_c} d x_1 
   =  {1 \over 2} (\int_{-1}^0 \mathbb{1}_{-|x_2| \leq y_c} d x_1 + \int_0^{1} \mathbb{1}_{|x_2| \leq y_c} d x_1) \notag \\
   & = {1 \over 2} (\mathbb{1}_{-|x_2| \leq y_c} + \mathbb{1}_{|x_2| \leq y_c}) = \begin{cases} 
{1 \over 2} \mathbb{1}_{-|x_2| \leq y_c} & \mbox{ if } y_c < 0 \\
{1 \over 2} (1+ \mathbb{1}_{|x_2| \leq y_c}) & \mbox{ if } y_c > 0 \\
\end{cases} \notag
 \end{align}
\paragraph{\textbf{Variance of conditional expectation}} $ \var[\esp[f|X_2]]$

$\esp[f|X_2]$ is independent of the sign of $x_2$ hence: $\displaystyle {1 \over 2} \int_{-1}^1 \esp[f|x_2] dx_2 = \int_0^1 \esp[f|x_2] dx_2= {1 + y_c \over 2}$

$\displaystyle {1 \over 2} \int_{-1}^1 \esp[f|x_2]^2 dx_2 = \int_0^1 \esp[f|x_2]^2 dx_2=$ $\begin{cases}
{1 \over 4} \int_0^1 \mathbb{1}_{-x_2 \leq y_c} d x_2 =  {1+y_c \over 4}& \mbox{ if } y_c < 0 \\
{1 \over 4} \int_0^1 (1+ 3 \cdot \mathbb{1}_{x_2 \leq y_c}) d x_2 = {1 + 3 y_c \over 4} & \mbox{ if } y_c > 0 \\
\end{cases}$

$\displaystyle \var[\esp[f|X_2]] = {1 \over 2} \int_{-1}^1 \esp[f|x_2]^2 dx_2 - ({1 \over 2} \int_{-1}^1 \esp[f|x_2] dx_2)^2={|y_c|-y_c^2 \over 4}$
\paragraph{\textbf{Sobol index} }
$\displaystyle S\!I_2^C(y_c)= {\var[\esp[f|X_2]] \over \var[f]} = {|y_c|-y_c^2 \over 1- y_c^2}= {|y_c| \over 1+|y_c|}$
\subsection*{A.2 Parameterization with two cutting values $y_{c1}$ and $y_{c2}$}
In this case, we study $f(x_1,x_2) =  \mathbb{1}_{sign(x_1)|x_2| \in [y_{c1},y_{c2}]} $ with $y_{c1} < y_{c2}$.
\subsubsection*{A.2.1 Derivation of $S\!I_1^C(y_{c1},y_{c2})$}
\paragraph{\textbf{Conditional expectation}}
$\esp[f|X_1] = {1 \over 2} \int_{-1}^{1} \mathbb{1}_{sign(x_1)|x_2| \in [y_{c1},y_{c2}]} d x_2$ 

If $x_1 < 0$: $\esp[f|X_1] = $
$\begin{cases}
0 & \mbox{ if } y_{c1} \geq 0 \\
{1 \over 2} (\int_{y_{c1}}^0 d x_2 +\int_0^{-y_{c1}} d x_2) = -y_{c1} & \mbox{ if } y_{c1} \leq 0 \leq y_{c_2} \\
y_{c2}-y_{c1} & \mbox{ if } y_{c2} \leq 0 \\
\end{cases}$

If $x_1 > 0$: $\esp[f|X_1] = $ 
$\begin{cases}
y_{c2}-y_{c1} & \mbox{ if } y_{c1} \geq 0 \\
{1 \over 2} \int_{-y_{c2}}^{y_{c2}} d x_2 = y_{c2} & \mbox{ if } y_{c1} \leq 0 \leq y_{c_2} \\
0 & \mbox{ if } y_{c2} \leq 0 \\
\end{cases}$
\paragraph{\textbf{Variance of conditional expectation}} $ \var[\esp[f|X_1]]$

$\displaystyle {1 \over 2} \int_{-1}^1 \esp[f|x_1] dx_1 = {1 \over 2} (\int_{-1}^0 \esp[f|x_1] dx_1+ \int_0^1 \esp[f|x_1] dx_1) = {y_{c2}-y_{c1} \over 2}$

$\displaystyle {1 \over 2} \int_{-1}^1 \esp[f|x_1]^2 dx_1 = {1 \over 2} (\int_{-1}^0 \esp[f|x_1]^2 dx_1+ \int_0^1 \esp[f|x_1]^2 dx_1) = $ 
$\begin{cases}
{(y_{c2}-y_{c1})^2 \over 2} & \mbox{ if } y_{c2} \leq 0 \mbox{ or } y_{c1} \geq 0 \\
{y_{c1}^2+y_{c2}^2 \over 2} & \mbox{ if } y_{c1} \leq 0 \leq y_{c_2} \\
\end{cases}$

Therefore $\displaystyle V[\esp[f|x_1]] = {1 \over 2} \int_{-1}^1 \esp[f|x_1]^2 dx_1 - ({1 \over 2} \int_{-1}^1 \esp[f|x_1] dx_1)^2={(|y_{c2}|-|y_{c1}|)^2  \over 4}$

\paragraph{\textbf{Variance }} $\var[f]=\esp[f^2]-\esp[f]^2$ with $\esp[f^2]= \esp[f]$ so: $\displaystyle \var[f] = {y_{c2}-y_{c1} \over 2} - ({y_{c2}-y_{c1} \over 2})^2= {y_{c2}-y_{c1} \over 4} (2-y_{c_2}+y_{c_1})$

\paragraph{\textbf{Sobol index}}
$\displaystyle SI_1^C(y_{c1},y_{c2})= {\var[\esp[f|x_1]] \over \var[f]} = {(|y_{c2}|-|y_{c1}|)^2 \over (y_{c2}-y_{c1}) (2-y_{c_2}+y_{c_1})}$

If $y_{c1}$ and $y_{c2}$ have the same sign:
$SI_1^C(y_{c1},y_{c2}) = {y_{c2}-y_{c1} \over 2-y_{c_2}+y_{c_1}}$ 

If $y_{c1}$ and $y_{c2}$ have opposite sign:
$SI_1^C(y_{c1},y_{c2}) = {(y_{c2}+y_{c1})^2 \over (y_{c2}-y_{c1}) (2-y_{c_2}+y_{c_1})}$

\subsubsection*{A.2.2 Derivation of $S\!I_2^C(y_{c1},y_{c2})$}
\paragraph{\textbf{Conditional expectation}}

\begin{align}
 \esp[f|X_2] & = {1 \over 2} \int_{-1}^{1} \mathbb{1}_{sign(x_1)|x_2| \in [y_{c1},y_{c2}]} d x_1  =  {1 \over 2} (\int_{-1}^0 \mathbb{1}_{-|x_2| \in [y_{c1},y_{c2}]} d x_1 + \int_0^{1} \mathbb{1}_{|x_2| \in [y_{c1},y_{c2}]} d x_1) \notag \\
   & = {1 \over 2} (\mathbb{1}_{-|x_2| \in [y_{c1},y_{c2}]} + \mathbb{1}_{|x_2| \in [y_{c1},y_{c2}]})  = \begin{cases} 
{1 \over 2} \mathbb{1}_{|x_2| \in [y_{c1},y_{c2}]} & \mbox{ if } y_{c1} \geq 0 \\
{1 \over 2} (\mathbb{1}_{-|x_2| \in [y_{c1},0]} + \mathbb{1}_{|x_2| \in [0,y_{c2}]}) & \mbox{ if } y_{c1} \leq 0 \leq y_{c_2} \\
{1 \over 2} \mathbb{1}_{-|x_2| \in [y_{c1},y_{c2}]} & \mbox{ if } y_{c2} \leq 0 \\
\end{cases} \notag
 \end{align}
 \paragraph{\textbf{Variance of conditional expectation}} $ \var[\esp[f|X_2]]$

$\esp[f|X_2]$ is independent of the sign of $X_2$ hence: $\displaystyle {1 \over 2} \int_{-1}^1 \esp[f|x_2] dx_2 = \int_0^1 \esp[f|x_2] dx_2= {y_{c2}-y_{c1} \over 2}$
\begin{align}
& {1 \over 2} \int_{-1}^1 \esp[f|x_2]^2 dx_2 = \int_0^1 \esp[f|x_2]^2 dx_2 \notag \\
     = & \begin{cases}
{1 \over 4} \mathbb{1}_{|x_2| \in [|y_{c1}|,|y_{c2}|]} = {y_{c2}-y_{c1} \over 4}  & \mbox{ if } y_{c2} \leq 0 \mbox{ or } y_{c1} \geq 0 \\
{1 \over 4} (\mathbb{1}_{-|x_2| \in [y_{c1},0]} + \mathbb{1}_{|x_2| \in [0,y_{c2}]}+ 2 \cdot  \mathbb{1}_{|x_2| \in [0,-y_{c1}]\cap [0,y_{c2}]})= {y_{c2}-y_{c1} +2 \min(|y_{c1}|,y_{c2})\over 4} & \mbox{ if } y_{c1} \leq 0 \leq y_{c_2} \\
\end{cases} \notag
\end{align}
Therefore
\begin{align}
 \var[\esp[f|X_2]] = & {1 \over 2} \int_{-1}^1 \esp[f|x_2]^2 dx_2 - ({1 \over 2} \int_{-1}^1 \esp[f|x_2] dx_2)^2 \notag \\
 = &
\begin{cases}
 {(y_{c2}-y_{c1})(1-y_{c2}+y_{c1}) \over 4}  & \mbox{ if } y_{c2} \leq 0 \mbox{ or } y_{c1} \geq 0 \\
 {(y_{c2}-y_{c1})(1-y_{c2}+y_{c1}) +2 \min(|y_{c1}|,y_{c2})\over 4} & \mbox{ if } y_{c1} \leq 0 \leq y_{c_2} \\
\end{cases}\notag
\end{align}

\paragraph{\textbf{Sobol index}}
$\displaystyle S\!I_2^C(y_{c1},y_{c2})= {\var[\esp[f|x_2]] \over \var[f]} $

$\displaystyle S\!I_2^C(y_{c1},y_{c2})=\begin{cases}
 {1-y_{c2}+y_{c1} \over 2-y_{c_2}+y_{c_1} }  & \mbox{ if } y_{c2} \leq 0 \mbox{ or } y_{c1} \geq 0 \\
 {1-y_{c2}+y_{c1} \over 2-y_{c_2}+y_{c_1} } + 2 {\min(|y_{c1}|,y_{c2})\over (y_{c_2}-y_{c_1})(2-y_{c_2}+y_{c_1})} & \mbox{ if } y_{c1} \leq 0 \leq y_{c_2} \\
\end{cases}$



\section*{Appendix B: Proof of Proposition \ref{prop_hist}}
We recall the expression of the discretized version of the first order sensitivity-based criterion used to performed the clustering based on single membership functions for an input $X$ whose conditional distribution given $Y \in C$ has an $n_x$ bins histogram noted $h^{C}$:
  $$\widetilde{S\!I}^C = {n_x \over \displaystyle \sum_{j=1}^{n_x} h_j^C (N- \sum_{j=1}^{n_x} h_j^C)} \sum_{i=1}^{n_x} (h_i^C - {1 \over n_x} \displaystyle \sum_{j=1}^{n_x} h_j^C)^2$$

Let us consider a cluster $C_0$ to which we either add a set of points with histogram  $h$ to form a cluster $C_1$ or two set of points with respective histograms $h$ and $h'=\theta h$ forming a cluster $C_2$. The criteria for these three sets are  $\widetilde{S\!I}^{C_0},\widetilde{S\!I}^{C_1},\widetilde{S\!I}^{C_2}$.

We suppose that adding to $C_0$ the set of points with histogram $h$ improves the clustering $(C_0,\overline{C_0})$, so that $\widetilde{S\!I}^{C_0} \leq \widetilde{S\!I}^{C_1}$. Then the principle of the proof is to show that there is a greater improvement in adding also the set with a perfectly correlated histogram $h'$, ie that $\widetilde{S\!I}^{C_1} < \widetilde{S\!I}^{C_2}$. Finally, supposing now then that there are two set $C_k$ and $C_k'$ with perfectly correlated histograms, there are necessarily of the same side of the optimal partition made by grouping elementary clusters, otherwise grouping $C_k$ and $C_k'$  would strictly improve the criterion, which is not possible. 

So we assume now that $\widetilde{S\!I}^{C_0} \leq \widetilde{S\!I}^{C_1}$ and want to show that  $\widetilde{S\!I}^{C_1} <  \ \widetilde{S\!I}^{C_2}$ :
\begin{align}
  \widetilde{S\!I}^{C_0} & =  {n_x \over \displaystyle \sum_{j=1}^{n_x} h_j^{C_0} (N- \sum_{j=1}^{n_x} h_j^{C_0})} \sum_{i=1}^{n_x} (h_i^{C_0} - {1 \over n_x} \displaystyle\sum_{j=1}^{n_x} h_j^{C_0})^2 \notag \\
  \widetilde{S\!I}^{C_1} & =  {n_x \over \displaystyle \sum_{j=1}^{n_x} (h_j^{C_0}+h_j) (N- \sum_{j=1}^{n_x}(h_j^{C_0}+h_j) )} \sum_{i=1}^{n_x} (h_i^{C_0}+h_i - {1 \over n_x} \displaystyle \sum_{j=1}^{n_x}(h_j^{C_0}+h_j ) )^2 \notag \\
  \widetilde{S\!I}^{C_2} & =  {n_x \over \displaystyle \sum_{j=1}^{n_x} (h_j^{C_0}+(1+\theta)h_j) (N- \sum_{j=1}^{n_x}(h_j^{C_0}+(1+\theta)h_j) )} \sum_{i=1}^{n_x} (h_i^{C_0}+(1+\theta)h_i - {1 \over n_x} \displaystyle
            \sum_{j=1}^{n_x}(h_j^{C_0}+(1+\theta)h_j) )^2 \notag
\end{align}

We denote: $\displaystyle N_0 = {N \over n_x}, H_0 = \displaystyle {1 \over n_x} \sum_{j=1}^{n_x} h_j^{C_0}, H_1 = {1 \over n_x} \sum_{j=1}^{n_x} h_j$ and $u=h^{C_0}-H_0, v = h-H_1$.
\begin{align}
  n_x \widetilde{S\!I}^{C_0} & =  {1 \over \displaystyle H_0 (N_0- H_0)} \sum_{i=1}^{n_x} u_i^2 \notag \\
  n_x  \widetilde{S\!I}^{C_1} & =  {1 \over \displaystyle (H_0+H_1) (N_0- (H_0+H_1) )} \sum_{i=1}^{n_x} (u_i+v_i)^2 \notag \\
  n_x  \widetilde{S\!I}^{C_2} & =  {1 \over \displaystyle (H_0+(1+\theta) H_1) (N_0- (H_0+(1+\theta) H_1) )} \sum_{i=1}^{n_x} (u_i+(1+\theta) v_i )^2 \notag
\end{align}
Assuming $\widetilde{S\!I}^{C_0} \leq \widetilde{S\!I}^{C_1}$ :
\begin{align}
  \label{uv}
  ||u+v||^2 \geq & {(H_0+H_1)(N_0- (H_0+H_1) ) \over  H_0 (N_0- H_0) }||u||^2
                   \end{align}
Assuming $\displaystyle H_0+(1+\theta) H_1 < N_0$, we consider $\displaystyle {\widetilde{S\!I}^{C_2} \over \widetilde{S\!I}^{C_1}}$:
\begin{align}
  {\widetilde{S\!I}^{C_2} \over \widetilde{S\!I}^{C_1}} = & \ {||u+(1+\theta)v||^2 \over ||u+v||^2} {(H_0+H_1) (N_0- (H_0+H_1) ) \over (H_0+(1+\theta) H_1)(N_0- (H_0+(1+\theta) H_1) )} \notag \\
   = & \ {||u+v||^2 +(\theta^2+2\theta) ||v||^2+ 2\theta u.v \over ||u+v||^2} {(H_0+H_1) (N_0- (H_0+H_1) ) \over (H_0+(1+\theta) H_1) (N_0- (H_0+(1+\theta) H_1) )} \notag
  \end{align}
  From $\displaystyle (\theta^2+2\theta)||v||^2+ 2\theta u.v  = \theta( ||u+v||^2-||u||^2 + (1+\theta)||v||^2) $ and Equation (\ref{uv}) we deduce:
\begin{align}
  (\theta^2+2\theta)||v||^2+ 2\theta u.v \geq & \theta (1- {H_0 (N_0- H_0) \over (H_0+H_1) (N_0- (H_0+H_1) )}) ||u+v||^2  \notag \\
  ||u+v||^2+  (\theta^2+2\theta)||v||^2+ 2 \theta u.v \geq & (1+\theta- {\theta H_0 (N_0- H_0) \over (H_0+H_1) (N_0- (H_0+H_1) )}) ||u+v||^2  \notag
\end{align}
Hence:
\begin{align}
  {\widetilde{S\!I}^{C_2} \over \widetilde{S\!I}^{C_1}} \geq & \ {(1+\theta) (H_0+H_1)(N_0- (H_0+H_1) )- \theta H_0 (N_0-H_0)   \over (H_0+(1+\theta) H_1) (N_0- (H_0+(1+\theta) H_1))} \notag \\
  {\widetilde{S\!I}^{C_2} \over \widetilde{S\!I}^{C_1}} \geq & \ {(H_0+(1+\theta) H_1) (N_0- (H_0+H_1) )- \theta H_0 H_1   \over (H_0+(1+\theta) H_1) (N_0- (H_0+H_1) )-\theta H_1 (H_0+(1+\theta) H_1)} \notag \\
  {\widetilde{S\!I}^{C_2} \over \widetilde{S\!I}^{C_1}} \geq & \ 1+ {\theta (1+\theta) H_1^2 \over (H_0+(1+\theta) H_1) (N_0- (H_0+(1+\theta)H_1) )} >1 \notag
\end{align}
so $\widetilde{S\!I}^{C_1} < \widetilde{S\!I}^{C_2}$.

\end{document}